\documentclass{article}%
\usepackage{amsmath}
\usepackage{amsfonts}
\usepackage{amssymb}
\usepackage{graphicx}
\usepackage{tikz,pgfplots}
\usepackage{epstopdf}
\usepackage{color}%
\providecommand{\U}[1]{\protect\rule{.1in}{.1in}}
\providecommand{\U}[1]{\protect\rule{.1in}{.1in}}
\newtheorem{theorem}{Theorem}

\newtheorem{definition}[theorem]{Definition}
\newtheorem{example}[theorem]{Example}
\newtheorem{lemma}[theorem]{Lemma}

\newenvironment{proof}[1][Proof]{\noindent\textbf{#1.} }{\ \rule{0.5em}{0.5em}}
\DeclareMathOperator{\spn}{span}

\DeclareMathOperator{\sinc}{sinc}

\DeclareMathOperator{\diag}{diag}

\begin{document}

\title{Approximation of Weakly Singular Integral Equations by Sinc Projection Methods}
\author{Khadijeh Nedaiasl\thanks{Institute for Advanced Studies in Basic Sciences,
Zanjan, Iran, e-mail: \texttt{ nedaiasl@iasbs.ac.ir, knedaiasl85@gmail.com}}}

\maketitle

\begin{abstract}
In this paper, two numerical schemes for a nonlinear integral equation of Fredholm type with weakly singular kernel are studied. These numerical methods blend collocation, convolution   approximations based on sinc basis functions with iterative schemes like steepest descent  and  Newton method that involve solving a nonlinear system of equations. The exponential rate of convergence for the convolution scheme is obtained and also the collocation methods are analyzed appropriately. Numerical experiments have been performed to illustrate the sharpness of the theoretical estimates and the sensitivity of the solution with respect to some parameters in the equation. The comparison between the schemes indicates that sinc-convolution method is more effective. 
\end{abstract} 

\vspace{1em} \noindent\textbf{Keywords:} Fredholm integral equation, Urysohn integral operator, weak singularity,  convolution method, collocation
method.

\vspace{1em} \noindent\textbf{Mathematics Subject Classification (2010):} 45B05, 45E99, 65J15, 65R60.

\section{Introduction}
The aim of this paper is  to study of the numerical solution of nonlinear Fredholm integral equation
\begin{equation}\label{asli}
u(t)= g(t) + \int^{b}_{a}  f(\vert t-s \vert) k(t,s)\psi (s, u(s))  \mathrm{d}s, \quad - \infty < a \leq t \leq b < \infty, 
\end{equation} 
where $u(t)$ is an unknown function to be determined and $k(t,s)$, $\psi(s,u)$ and  $g(t)$ are given functions. Eq. \eqref{asli} is an algebraic weakly singular integral equation, whenever $f(t)$  is defined by $t^{-\lambda}$, $0<\lambda<1$. 
A more general type of this equation, so-called Urysohn weakly singular integral equation \cite{urysohn}, is defined as
 \begin{equation}\label{uri}
 u(t)= g(t) + \int^{b}_{a}  f(\vert t-s \vert) k(t,s, u(s))  \mathrm{d}s, \quad - \infty < a \leq t \leq b < \infty. 
 \end{equation}  
Linear and nonlinear integral equations with weakly singular kernels arise in various applications such as astrophysics \cite{amosov}. In the potential theory, the boundary integral equations of the Laplace and Helmholtz operators are described as a linear combination of weakly singular operators \cite{Nedelec}.

It is well-known that the solution of Eq. \eqref{asli} has some singularities near the boundaries. This is an important note that should be considered in the designation of the numerical method. There is a considerable interest in the numerical analysis of linear and nonlinear integral equations with weakly singular kernels. 
This interest has been followed by some projection schemes such as Galerkin, collocation and product integration methods with singularity preserving approaches  which find an approximation with  optimal error bound \cite{scientific, mandal2019superconvergence, allouch, cao1994, zhong, GRAMMONT2016309, chen}. It is worth mentioning that the numerical solution of Eq. \eqref{asli} with the smooth kernel is comprehensively studied, for more information see \cite{atkinson1992survey,  may12, neda}.

{In the current study, we propose two reliable schemes in order to achieve appropriate approximations for the nonlinear weakly singular  integral equations. The methods are designed such that compensate the singular behavior of the solution. 
	For the sake of good comparison, we attempt to present two algorithms based on sinc approximation method. In what follows, we will elucidate these schemes with relevant characteristics and convergence rates. }
 
 The first objective of this
 study is to investigate the analysis of sinc-collocation method for nonlinear weakly singular Fredholm integral equation.  In \cite{MALEKNEJAD20151}, the authors have studied this subject and obtained the rate of convergence $\mathcal{O}(\Vert A^{-1} \Vert_{\infty}(3+\log(N))\sqrt{N}\exp(-\sqrt{\pi d \lambda N}))$. Here, we consider the sinc-collocation with a different basis functions by adding two fractional polynomials $(t-a)^{\lambda}$ and $(b-t)^{\lambda}$ in the finite dimensional space utilized as the  approximation space.  We propose an error analysis for the approximate solution chosen from an appropriate finite dimensional space built by the shifted sinc functions and considering the singularity properties of the exact solution.  However, we encounter a term $\xi_{N}$ in the upper bound which depends on $N$ and is unavoidable due to the nature of the projection methods.

For the second objective, we present and analyze a numerical scheme using 
 the key idea of how to approximate the following nonlinear convolution 
\[
r(x) = \int_{a}^{x} k(x, x-t, t)\mathrm{d}t, 
\] 
 which is called sinc-convolution method. Here, we replace the variables with the single exponential transformation introduced in Section \ref{pel}. 
  
In order to make the paper self-contained, the basic properties of the sinc approximation  method are introduced in Section \ref{pel}. Two numerical schemes based on sinc-collocation and -convolution methods will be studied in Section \ref{twoscheme}. Furthermore, this section contains a complete convergence analysis for the proposed methods. Finally, Section \ref{numeric} is devoted to some numerical experiments in order to show the consistency with the theoretical estimates of the convergence rate.  

 In this work, we present numerical schemes based on sinc approximation; sinc-convolution and  sinc-collocation methods  for the nonlinear Fredholm weakly singular integral equations. Sinc-convolution is introduced in \cite{stenger1995collocating} to collocate indefinite integrals of convolution type, and it could be interpreted as a special type of Nystr\"{o}m method. 
 It will be shown that this method has the exponential rate of convergence. For a comprehensive study of sinc-convolution method and its applications to different kinds of equations, we refer to \cite{stenger2012numerical, stenger2016handbook}. Furthermore, sinc-collocation method and its features for nonlinear integral equation are studied in this paper. 
Equations (\ref{asli}) and (\ref{uri}) can be expressed in the operator form as
 \begin{equation}\label{asli22}(I-\mathcal{K}_{i})u=g, \quad i=1,2,\end{equation} 
 where $(\mathcal{K}_{1}u)(t)=\int_a^b f(| t-s|)k(t,s) \psi (s,u(s))\mathrm{d}s$ and $(\mathcal{K}_{2}u)(t)=\int_a^b f(| t-s|)k(t,s,u(s))\mathrm{d}s$. 
 {The operators are defined on the Banach space $X=\textbf{H}^{\infty}({\mathcal{D}})\bigcap C(\bar{\mathcal{ D}})$. In this notation, $\mathcal{ D}\subset\mathbb{C}$ is a simply connected domain which satisfies $(a,b)\subset\mathcal{ D}$ and $\textbf{H}^{\infty}(\mathcal{D})$ denotes the family of all functions $f$ that are analytic in the domain $\mathcal{ D}$ and  have finite the uniform (superimum) norm.}
We suppose that the unknown solution $u(t)$ to be determined is geometrically isolated \cite{krasnosel2012approximate, MALEKNEJAD20113292}, which means that there is some balls \[ \mathfrak{B}(u,r) =\{x \in X : \|x-u\| \leq  r\},\] with
$r>0$ and Eq. \eqref{asli} has  the only solution $u$. 

\section{Preliminaries}\label{pel}
In order to make the paper self-contained, some basic definitions and theorems on sinc function, sinc interpolation and quadrature are presented.
\subsection{Sinc interpolation}
 The sinc function on the real line, $\mathbb{R}$ is defined by 
 \begin{equation*}\sinc
(t)=\left\{\begin{array}{lr}\frac{\sin(\pi t)}{\pi t},& t\neq0,\\
1, & t=0.\end{array}\right.\end{equation*}
It is well-known that  a function $f$ with suitable smoothness properties could be approximated by sinc functions as
\begin{equation}\label{tagrib}f(t)\approx \sum_{j=-N}^N f(jh)S(j,h)(t),\quad
t\in\mathbb{R}, \end{equation}
wherein the basis function $S(j,h)(t)$
is given  by
\begin{equation}\label{sinc function}S(j,h)(t)=\sinc (\frac{t}{h}-j), \quad j \in \mathbb{Z}, \end{equation}
and $h$ is a step size appropriately picked depending on a given positive
integer $N$,  and \eqref{sinc function} is called $j
\textrm{th}$ sinc function.
Eq. \eqref{tagrib} is adjusted to approximate
on general intervals with the accompanying a  variable transformations
$t=\varphi(x)$. 
Appropriate single exponential and double exponential  transformations could be applied \cite{stenger2012numerical, tanaka} as the converting function $\varphi(x)$. The single exponential transformation and its inverse is given as below
\begin{equation*}\varphi_{a,b}
(x)=\frac{b-a}{2}\tanh(\frac{x}{2})+\frac{b+a}{2},
\end{equation*}
\begin{equation*}\phi_{a,b}(t)=\log(\frac{t-a}{b-t}),\end{equation*}
respectively. 
The superscripts $a$ and $b$ in the transformations notation play an important role in the application of sinc-collocation method for weakly singular integral equations.
 The strip domain is introduced in order to define a suitable function space 
\[\mathcal{D}_{d}=\big\{z\in \mathbb{C } : |\Im z|<d\big\},\] for
some $d>0$. When it is incorporated with this
transformation, then we consider the
transformed domain
\begin{equation*}\varphi(\mathcal{D}_{d})=\Big\{z\in \mathbb{C } :\Big|\arg(\frac{z-a}{b-z})\Big|<d\Big\}.\end{equation*}
 The following definitions and theorems are treated for the details  of the procedure.
%
\begin{definition}(\cite{stenger2012numerical})
Let $\alpha$ and  $C$  be  positive constants, and let $\mathcal{D}$ be a bounded and simply connected domain which satisfies $(a,b)\subset \mathcal{D}$. Then
$\mathcal{L}_{\alpha}(\mathcal{D})$ denotes the set of all functions $f\in \mathbf{H}^{\infty}(\mathcal{D})$ which satisfy
\begin{equation}\label{L}
|f(z)|\leq C{|Q(z)|}^{\alpha},\end{equation} for all $z$ in $\mathcal{D}$ where $Q(z)=(z-a)(b-z)$.
\end{definition}
The next theorem clarify the exponential convergence rate of the sinc
interpolation.
\begin{theorem}\label{interpolationse}(\cite{Okayamanumerische}) Let $f\in \mathcal{L}_{\alpha}(\varphi_{a,b}(\mathcal{D}_{d}))$ for $d$ with
$0<d<\pi$.
Suppose that 
$h$ be given by the formula $h=\sqrt{\frac{\pi d}{\alpha N}}$, wherein $N$ is a positive integer. 
Then there is a constant $C$ independent of $N$, such that
\begin{equation*}
\Big \Vert f(t)-\sum_{j=-N}^N
f(\varphi_{a,b}(jh))S(j,h)(\phi(t))\Big \Vert \leq C \sqrt{N}
\exp(-\sqrt{\pi d \alpha N}), 
\end{equation*}
where 
\[C=\frac{2K(b-a)^{2\alpha}}{\alpha}\Big[ \frac{2}{\pi d (1-e^{-2\sqrt{\pi d \alpha}})(\cos(\frac{d}{2}))^{2\alpha}}+\sqrt{\frac{\alpha}{\pi d}} \Big].\]
\end{theorem}
According to Theorems \ref{interpolationse}, in order to attain an exponential convergence, the approximated function should exist in $\mathcal{L}_{\alpha}(\mathcal{D})$. By the condition \eqref{L}, such a function is expected to be zero
at the endpoints, which is too  restrictive in practice.
However, it can be reduced to the following function space $\mathcal{M}_{\alpha}(\mathcal{D})$
with $0<\alpha\leq1$ and
$0<d<\pi$.
\begin{definition}(\cite{stenger2012numerical}) Let $\mathcal{D}$ be a simply connected and bounded domain which contains $(a,b)$.
The family $\mathcal{M}_{\alpha}(\mathcal{D})$ contains all  analytical
functions which are continuous on $\bar{\mathcal{D}}$ such that the
transformation
\begin{equation*}G[f](t)=f(t)-[(\frac{b-t}{b-a})f(a)+(\frac{t-a}{b-a})f(b)],\end{equation*}
resides in $\mathcal{L}_{\alpha}(\mathcal{D})$.\end{definition}
 \subsection{Sinc quadrature}
Sinc approximation by incorporating with single exponential  transformation could be applied to definite integration based on the function
 approximation to designate  the sinc quadrature. The following theorem includes error bound for the sinc quadrature of $f$ on $(a,b)$.

\begin{theorem}\label{quadrature se}(\cite{Okayamanumerische})
Let $(fQ)\in \mathcal{L}_{\alpha}(\varphi_{a,b}(\mathcal{ D}_{d}))$ for $d$ with
$0<d<\pi$. Suppose that
$N$ be a
positive integer and $h$ is selected by the formula
$$h=\sqrt{\frac{\pi d}{\alpha N}}.$$
Then 
\begin{equation}\label{quadrature se}
\Big\vert \int_a^b f(s)\,\mathrm{d}s-h \sum_{j=-N}^
Nf(\varphi_{a,b}(jh))(\varphi_{a,b})'(jh)\Big\vert \leq C (b-a)^{2\alpha -1}\exp (-\sqrt{\pi d \alpha N}),
\end{equation}
where $C$ is a constant independent of $N$.
\end{theorem}

\section{Two numerical schemes}\label{twoscheme}

\subsection{Sinc-collocation}
In this section, the sinc-collocation and its aspects for the nonlinear Fredholm integral equation with weakly singular kernel are  discussed. 
A sinc approximation
$u_{N}$ to the solution $u\in
\mathcal{M}_{\lambda}(\varphi_{a,b}(\mathcal{ D}_{d}))$
of Eq. \eqref{asli} is constructed in this subsection. For this aim the interpolation operator
{$\mathcal{P}_{N} :\mathcal{M}_{\lambda}\rightarrow X_{N}$} is defined as 
$$\mathcal{P}_{N} [u](t)=\mathfrak{L}u(t)+\sum_{j=-N}^N
[u(t_{j} )-(\mathfrak{L}u)(t_{j})] S(j,h)(\phi_{a,b} (t)),$$
where
$$\mathfrak{L}[u](t)=(\frac{b-t}{b-a})^{\lambda}u(a)+(\frac{t-a}{b-a})^{1-\lambda}u(b).$$
In this formula, the sinc points
$t_{j}$
are determined by 
\begin{equation}\label{sincpoints}
t_{j}=\left\{\begin{array}{ll}a, & \hbox{$j=-N-1$,}  \\
\varphi_{a,b} (jh), & \hbox{$j=-N,\ldots,N$,}\\ b, & \hbox{$j=N+1$.}
\end{array}\right.
\end{equation}
The approximate solution could be represented as
\begin{equation}
u_{N}(t)=c_{-N-1}(\frac{b-t}{b-a})^{\lambda}+\sum_{j=-N}^{N}
c_{j}S(j,h)(\phi_{a,b}(t))+c_{N+1}(\frac{t-a}{b-a})^{1-\lambda},
\end{equation}
where the singularity  exponent parameter $\lambda$ is introduced in Eq. (\ref{asli}).
It is worthy to notice that the choice of these basis functions incorporate with sinc function reflects the singularity of the exact solution well. 
Employing the operator $\mathcal{P}_{N}$ to both sides of Eq. \eqref{asli}
leads us the following approximate equation 
\begin{equation}\label{collocation}u_{N}= \mathcal{P}_{N}g + \mathcal{P}_{N}\mathcal{K}u_{N}.\end{equation}
This equation could be rewritten as 
\begin{equation}\label{sinc col}
u_{N}(t_{i}) =g(t_{i}) +\int_{a}^{b} f(\vert t_{i} -s \vert)k(t_{i},s)\psi(s, u_{N}(s))\mathrm{d}s, \quad i=-N-1, \dots, N+1,
\end{equation}
so the collocation method  for solving Eq. \eqref{asli} aggregates with  \eqref{sinc col} for $N$ sufficiently large.
We utilize the theory of holomorphic function space along with the singularity preserving representation of the approximate solution to blend a  mechanism for approximating the singular integrals which arise from the discretization of weakly singular integral operators. 
Let us have the following presentation of  Eq. (\ref{sinc col}):
\begin{equation}\label{sinc col1}
\begin{split}
u_{N}(t_{i}) =&\int_{a}^{t_{i}} f(\vert t_{i} -s \vert)k(t_{i}, s)\psi(s, u_{N}(s))\mathrm{d}s\\
+& \int_{t_{i}}^{b} f(\vert t_{i} -s \vert)k(t_{i},s)\psi(s, u_{N}(s))\mathrm{d}s+ g(t_{i}), \quad i= -N-1, \dots, N+1.
\end{split}
\end{equation}
Due to the complexity of the integral kernel, we utilize the approximation of the integral operator in \eqref{sinc col1} by the
quadrature formula presented in \eqref{quadrature se}. We notice that in order to use the sinc quadrature method properly, the intervals $[a,t_{i} ]$ and $[t_{i}, b] $ should be transformed to the whole real line. So, Eq. \eqref{sinc col1} could be written as
\begin{equation}\label{coldis}
\begin{split}
u_{N}(t_{i}) =&h\vert t_{i} -a  \vert^{\lambda}  \sum_{j=-N}^{N} \frac{1}{(1+e^{jh})^\lambda(1+e^{-jh})}k\big(t_{i},\varphi_{a, t_{i}}(jh)\big)
\psi\big(\varphi_{a, t_{i}}(jh), u_{N}(\varphi_{a, t_{i}}(jh))\big)
+ \\
&h\vert b- t_{i}   \vert^{\lambda}  \sum_{j=-N}^{N} \frac{1}{(1+e^{jh})^\lambda(1+e^{-jh})}k\big(t_{i},\varphi_{t_{i}, b}(jh)\big)
\psi\big(\varphi_{t_{i},b}(jh), u_{N}(\varphi_{ t_{i},b}(jh))\big)+\\
&g(t_{i}),\quad  i= -N-1, \dots, N+1.
\end{split}
\end{equation}
This numerical procedure points us to change \eqref{coldis} with operator notation
\begin{equation}\label{10}
u_{N}-\mathcal{P}_{N}\mathcal{K}_{N}u_{N}=\mathcal{P}_{N}g,
\end{equation}
where the discrete operator $\mathcal{K}_{N}u$ is defined as 
\begin{equation}
\begin{split}
(\mathcal{K}_{N}u)(t):=& h\vert t -a  \vert^{\lambda}  \sum_{j=-N}^{N} \frac{1}{(1+e^{jh})^\lambda(1+e^{-jh})}k\big(t, \varphi_{a, t_{i}}(jh)\big)\psi\big(\varphi_{a, t_{i}}(jh), u(\varphi_{a, t_{i}}(jh))\big)
 \\
+&h\vert b- t   \vert^{\lambda}  \sum_{j=-N}^{N} \frac{1}{(1+e^{jh})^\lambda(1+e^{-jh})}k\big(t,\varphi_{ t_{i},b}(jh)\big)
\psi\big(\varphi_{t_{i},b}(jh), u(\varphi_{ t_{i},b}(jh))\big).
\end{split}
\end{equation}
The Eq. \eqref{10} is the operator form of discrete collocation method based on the sinc basis function. 
By solving the nonlinear system of equations (\ref{10}), the unknown
coefficients in $u_{N}$ are determined. 

\subsubsection{Convergence analysis}
In this subsection we give an error analysis for the sinc-collocation method.
We state the following lemmas which are used subsequently.
\begin{lemma}\label{franks} (\cite{stenger2012numerical}) Let $h>0$. Then it holds that
	\begin{equation}\sup_{x\in\mathbb{R}}\sum_{j = -N}^{N}|S(j,h)(x)|\leq
	\frac{2}{\pi}(3+\log (N)).
	\end{equation}
\end{lemma}
This lemma concludes that
$\|\mathcal{P}_{N}\|\leq C \log(N)$ where $C$ is a constant independent of $N$ and $\mathcal{P}_{N}$ is the interpolation operator constructed on the sinc points.
\begin{lemma}(\cite{okayamacam})\label{lem}
	Let $ d $ be a constant with $ 0<d<\pi $. Define a function $ \varphi_{1} $ as 
	\[
	\varphi_{1}(x)=\frac{1}{2}\tanh(\frac{x}{2})+\frac{1}{2}.
	\]
	Then there is a constant $ c_{d} $ such that for all $ x \in \mathbb{R} $ and $ y \in [-d, d] $,
	\begin{equation}\label{A1}
	\vert \{ \varphi_{a,b}\}'(x+i y)\vert \leq (b-a) c_{d} \varphi_{1}'(x),
	\end{equation}
	\begin{equation}\label{A2}
	\vert \varphi_{0,1}(x+i y)\vert \geq \varphi_{1}(x).\end{equation}
	In addition, if $ t \leq x $,
	\begin{equation}\label{A3}
	\vert \varphi_{a,b}(x + iy) -  \varphi_{a,b} (t + iy) \vert \geq (b-a) \{\varphi_{1}(x) - \varphi_{1}(t)\}.
	\end{equation}
\end{lemma}
With the aid of Lemma \ref{lem}, the analytical behavior of the solution is investigated for a general kernel function.
It is convenient to define the following nonlinear operators which will be used  in the next theorem
\begin{equation}\label{operators}\begin{split}
(\mathcal{K}^{1}u)(t) & =  \int_a^t \vert t-s \vert ^ {-\lambda} k(t,s,u(s))\mathrm{d}s, \\
(\mathcal{K}^{2}u)(t) & = \int_t^b \vert t-s \vert ^ {-\lambda} k(t,s,u(s))\mathrm{d}s.
\end{split}\end{equation}
\begin{theorem}\label{theo3}
	Let $ \mathcal{D} =( \varphi_{a,b})^{-1}(\mathcal{D}_{d}) $ for a constant  $ d $  with $ 0<d< \pi $ and . Suppose that $ k(z,. , v) \in \mathbf{H}^{\infty}(\mathcal{D}) $ for all $ z $ and $ v $ belong to $ \overline{\mathcal{D}}$, and $ k(z,w , .) \in \mathbf{H}^{\infty}(\mathcal{D}) $ for all $ z $ and $ w $ belong to $ \overline{\mathcal{D}} $. Moreover, $ k(. , v, w) \in \mathcal{M}_{1-\lambda}(\mathcal{D}) $ for all $ v , w \in \overline{\mathcal{D}} $, $ k(z, v, w) $ is bounded for $ z, v $ and $ w $ in $ \overline{\mathcal{D}} $ and $ y \in \mathcal{M}_{\beta}(\mathcal{D}) $. Then  the solution $ u $ of (\ref{asli}) belongs to $ \mathcal{M}_{\gamma}(\mathcal{D}) $, where $ \gamma = \min (1 - \lambda , \beta) $.
\end{theorem}
\begin{proof}
	In \cite[p. 83]{zaber}, sufficient conditions have been mentioned to have a nonlinear analytic operator and then analytic solution. So, it is adequate to show that $ \mathcal{K}u $ is ($ 1-\lambda $)-H\"{o}lder continuous. For this aim, we show the operators defined in (\ref{operators}) have this property. To proof the ($ 1-\lambda $)-H\"{o}lder continuity of $ \mathcal{K}^{1}u $ and $ \mathcal{K}^{2}u $, the idea of Lemma A.2. in \cite{okayamacam} is extended to the nonlinear case. Set $ x = \mathrm{Re} [(\varphi_{a, b})^{-1}(z)] $, $ y =  \mathrm{Im} [(\varphi_{a, b})^{-1}(z)] $ and $ v = \varphi_{a, b}(t + i y)$ as a variable transformation,
	\[\begin{split}
	(\mathcal{K}^{1}u)(z) -& (\mathcal{K}^{1}u)(a) = \int^z_a \vert z-v \vert ^ {-\lambda} k(z, v , u(v)) \mathrm{d}v - 0  \\
	&=\int^x_{\infty} \vert \varphi_{a,b}(x + i y) - \varphi_{a,b}(t + i y) \vert^ {-\lambda} k(x + i y , t + i y, u(t + i y)) (\varphi_{a,b})'(t + i y) \mathrm{d}t.
	\end{split}\]
	Applying the absolute value on both sides of the above equation and using Eqs. \eqref{A1} and \eqref{A3}, we have 
	\[\begin{split}
	\vert (\mathcal{K}^{1}u)(z) - (\mathcal{K}^{1}u)(a) \vert \leq & \int^x_{\infty} ( b-a )^{- \lambda} (\varphi_{1}(x) - \varphi_{1}(t))^{-\lambda} M_{k}(b-a) c_{d} \varphi'_{1}(t)\mathrm{d}t\\
	\leq & \frac{M_{k} c_{d}}{1-\lambda}((b-a)\varphi_{1}(x))^{1-\lambda},
	\end{split}\]
	where $ M_{k} = \max_{\overline{\mathcal{D}}} \vert k(z,w,v)\vert  $.
	In addition, by using (\ref{A2}), the inequality $ (b-a)\varphi_{1}(x) \leq \vert z-a \vert $ is concluded. So 
	\begin{equation}
	\vert (\mathcal{K}^{1}u)(z) - (\mathcal{K}^{1}u)(a) \vert \leq \frac{M_{k} c_{d}}{1-\lambda} \vert z-a \vert^{(1-\lambda)}.
	\end{equation}
	Now, the $ (1-\lambda) $-H\"{o}lder continuity at the point $ b $ is considered
	\[\begin{split}
	(\mathcal{K}^{1}u)(b) - (\mathcal{K}^{1}u)(z) = & \int_{a}^{b} \vert b-v \vert^{-\lambda}\Big\{ k(b, v, u(v)) - k(z, v, u(v))\Big\} \mathrm{d}v\\
	+ & \int_{a}^{b}\Big\{\vert b-v \vert ^{-\lambda} - \vert z-v \vert ^{-\lambda}\Big\}k(z, v, u(v)) \mathrm{d}v\\
	- & \int_{b}^{z} \vert z-v \vert ^{-\lambda} k(z, v, u(v)) \mathrm{d}v.
	\end{split}\]
	Since $ k(., v, w) \in \mathcal{M}_{1-\lambda}(\mathcal{D})$, there exists $ M_{1} $ such that 
	\[\begin{split}
	\Big\vert \int_{a}^{b} \vert b-v \vert ^{-\lambda} \Big\{ k(b, v, u(v)) - k(z, v, u(v))\Big\} \mathrm{d}v\Big\vert  \leq & M_{1}\vert b-z \vert^{(1-\lambda)}\int_{a}^{b} \vert b-v \big \vert ^{-\lambda}\mathrm{d}v\\
	\leq & \frac{M_{1} \vert b-a \vert ^{1-\lambda}}{1-\lambda}\vert b-z\vert^{1-\lambda}.
	\end{split}\]
	The third term is bounded by
	\begin{equation}
	\big\vert \int_{b}^{z} \vert z-v \vert^{-\lambda}k(z, v, u(v))\mathrm{d}v \big\vert \leq \frac{M_{k} c_{d}}{1-\lambda}\vert b-z\vert^{1 - \lambda}.
	\end{equation}
	Integration by part, the H\"{o}lder continuity of the function $F(z)=z^{1-\lambda}$ and the assumptions on $k(z, w, .)$ and 
	$k(z, . , v)  \in  \mathbf{H}^{\infty}(\mathcal{D})$ result that
	\[\vert (\mathcal{K}^{1}u)(b) - (\mathcal{K}^{1}u)(z) \vert \leq \frac{M_{2}}{1-\lambda}\vert b-z\vert^{(1-\lambda)} .\]
	The $(1-\lambda)$-H\"{o}lder continuity of the operator $ \mathcal{K}^{2}(u) $ can be proved in a similar manner. 
\end{proof}

 The Fr\'{e}chet derivative of the nonlinear operators $\mathcal{K}$ and  $\mathcal{K}_{N}$ for all $u$ is stated by
\begin{equation} 
\begin{split}
(\mathcal{K}'u)x(t) = \int_{a}^{b} f(\vert t-s \vert) k(t,s) \frac{\partial \psi}{\partial u}(s,u(s))x(s)\mathrm{d}s, \quad t\in [a,b], \quad x \in X,
\end{split}
\end{equation} 
and 
\begin{equation}
\begin{split}
(\mathcal{K'}_{N}u)x(t)=& h\vert t -a  \vert^{\lambda}  \sum_{j=-N}^{N} \frac{1}{(1+e^{jh})^\lambda(1+e^{-jh})}k\big(t, \varphi_{a, t_{i}}(jh)\big)\frac{\partial \psi}{\partial u}\big(\varphi_{a, t_{i}}(jh), u(\varphi_{a, t_{i}}(jh))\big)x(jh)
\\
+&h\vert b- t   \vert^{\lambda}  \sum_{j=-N}^{N} \frac{1}{(1+e^{jh})^\lambda(1+e^{-jh})}k\big(t,\varphi_{ t_{i},b}(jh)\big)
\frac{\partial \psi}{\partial u}\big(\varphi_{t_{i}, b}(jh), u(\varphi_{t_{i},b}(jh))\big)x(jh).
\end{split}
\end{equation}
\begin{theorem} Assume that $u(t)$ is the true
		solution of Eq. \eqref{asli} such that $I-\mathcal{K}'u$ is
		 a non-singular operator. Also, the term $\frac{\partial^{2}\psi}{\partial u^{2}}(t,s,u)$ is well-defined  and continuous on its domain. Furthermore, assume that $g\in \mathcal{M}_{\lambda}(\varphi_{a,b}(\mathcal{ D}_{d}))$ and
		$\mathcal{K}u\in \mathcal{M}_{\lambda}(\varphi_{a,b}(\mathcal{ D}_{d})) $ for
		all $u\in \mathfrak{B}(u,r)$.
		Then, there is a constant $C$ independent of $N$ such that
		\begin{equation}\label{q}\|u-u_{N}\|\leq C \xi_{N}\sqrt{N}
		\log(N+1)\exp(-\sqrt{{\pi} d \lambda N}),\end{equation}
		where $\xi_{N}=\|(I-\mathcal{P}_{N}(\mathcal{K}_{N})'(u))^{-1}\|.$
\end{theorem}
\begin{proof}
To find an upper error bound, we subtract \eqref{asli22} from \eqref{10} and obtain 
\begin{equation*}
u-u_{N}=\mathcal{K}u-
\mathcal{P}_{N}\mathcal{K}_{N}u_{N}+g-\mathcal{P}_{N}g.
\end{equation*}
The aforementioned relation is rewritten as 
\begin{equation}\begin{array}{cl} u-u_{N} & =
(I-\mathcal{P}_{N}(\mathcal{K}'_{N})(u))^{-1}\big\{(g-\mathcal{P}_{N}g) \\ & \\
& +(\mathcal{K}u-\mathcal{P}_{N}\mathcal{K}u)
+\mathcal{P}_{N}(\mathcal{K}u-\mathcal{K}_{N}u)\\ & \\
& +
\mathcal{P}_{N}
(\mathcal{K}_{N}u-\mathcal{K}_{N}u_{N}
-(\mathcal{K'}_{N})(u)(u-u_{N})) \big\}.
\end{array}\end{equation}
Finally, the following relation is obtained 
 \begin{equation}\begin{array}{cl} \|u-u_{N}\| & \leq
\|(I-\mathcal{P}_{N}(\mathcal{K}'_{N})(u))^{-1}\|\big\{\|g-\mathcal{P}_{N}g\| \\ & \\
& +\|\mathcal{K}u-\mathcal{P}_{N}\mathcal{K}u\|
+\|\mathcal{P}_{N}\|\|\mathcal{K}u-\mathcal{K}_{N}u\|\big\}
+\|\mathcal{P}_{N}\|\mathcal{O}(\|u-u_{N}\|^{2}).
\end{array}
\end{equation}
Because of $g, \mathcal{K}u\in
\mathcal{M}_{\lambda}(\varphi_{a,b}(\mathcal{ D}_{d}))$ and Theorem \ref{interpolationse}, we obtain
$$\|g-\mathcal{P}_{N}g\|\leq C_{1}\sqrt{N} \exp(-\sqrt{\pi d \lambda N}),$$
$$\|\mathcal{K}u-\mathcal{P}_{N}\mathcal{K}u\|\leq C_{2}\sqrt{N} \exp(-\sqrt{\pi d \lambda N}).$$
By using Theorem \ref{quadrature se}, we conclude that 
$$\|\mathcal{K}u-\mathcal{K}_{N}u\|\leq C_{3}
\exp(-\sqrt{\pi d \lambda N}),$$
and in final, we find un upper bound for  $\|\mathcal{P}_{N}\|$ by Lemma \ref{franks}.
Hence, $$\|u-u_{N}\|\leq C \xi_{N} \log(N+1)\sqrt{N}\exp(-\sqrt{\pi d \lambda N}).$$
\end{proof}
\subsection{Sinc-convolution}
Let $f(t)$  be a function with singularity at the origin and $g(t)$ be a function with singularities at both endpoints. The method of sinc-convolution is based on an accurate approximation of following integrals
\begin{equation}\label{convints}
\begin{split}
p(s)=\int^{s}_{a} f(s-t)g(t)\mathrm{d}t, \quad  s \in (a,b),\\
q(s)=\int^{b}_{s} f(t-s)g(t)\mathrm{d}t,\quad s \in (a,b),
\end{split}
\end{equation}
and then, it  could be used to approximate the definite convolution integral
\begin{equation}
\int^{b}_{a}f(\vert s-t \vert)g(t)\mathrm{d}t.
\end{equation}
In order to make it sense, the following notations are defined. 

\begin{definition}\label{def2}
	For a given positive integer  $ N $, let $ D_{N} $ and $ V_{N} $ denote linear operators acting on function $ u $ by 
	\begin{equation}\label{diag}\begin{split}
	D_{N}u=& \diag[u(t_{-N}), \ldots, u(t_{N})] ,\\ 
	V_{N}u= & (u(t_{-N}), \ldots, u(t_{N}))^{T},
	\end{split}\end{equation}
	where the superscript $ T $ specifies the transpose and $ \diag $ symbolizes the diagonal matrix. Set the basis functions as follow
	\begin{equation}
	\begin{split}
	\gamma_{j}(t)&=S(j,h)(\varphi_{a,b}(t)), \quad j=-N, \ldots, N,\\
	\omega_{j}(t)&=\gamma_{j}(t),        \quad  j= -N, \ldots , N,\\
	\omega_{-N}(t)&=\frac{b-t}{b-a} - \sum_{j= -N+1}^{N}\frac{1}{1+e^{jh}}\gamma_{j}(t),\\
	\omega_{N}(t)&=\frac{t-a}{b-a} - \sum_{j= -N}^{N-1} \frac{e^{jh}}{1+e^{jh}}\gamma_{j}(t).
	\end{split}\end{equation} 
\end{definition}
With the aid of these basis functions for a given vector $ \textbf{c}=(c_{-N}, \ldots, c_{N})^{T} $, we consider   a linear combination symbolized by $ \Pi_{N}$ as follows
\begin{equation}
(\Pi_{N}\textbf{c})(t)=\sum_{j=-N}^{N}c_{j}\omega_{j}(t).
\end{equation}
Let us define the interpolation operator $ \mathcal{P}^{c}_{N}:\mathcal{M}_{\lambda}(\mathcal{D}) \rightarrow X_{N}=\spn\{\omega_{j}(t)\}_{j=-N}^{N} $ as follows 
\[ \mathcal{P}^{c}_{N}f(t) = \sum_{j=-N}^{N}f(t_{j}) \omega_{j}(t), \]
where $ t_{j} $s are defined in (\ref{sincpoints}).
The numbers $ \sigma_{k} $ and $ e_{k} $ are determined by 
\begin{equation}\label{sigma}
\begin{split}
\sigma_{k}= & \int^{k}_{0} \sinc(t) \mathrm{d}t, \quad k \in \mathbb{Z},\\ 
e_{k}= & \frac{1}{2}+\sigma_{k}.
\end{split}
\end{equation}
Set an $ (2N+1) \times (2N+1) $ (Toeplitz) matrix $ I^{(-1)}=[e_{i-j}] $ where $ e_{i-j} $ represents the $ (i,j)^{th} $ element of $ I^{(-1)} $. In addition, the operators $ \mathcal{I}^{+} $ and  $ \mathcal{I}^{-} $ are specified as follows
\begin{equation}\label{sigma}\begin{split}
(\mathcal{I}^{+} g)(t) & = \int_{a}^{t}g(s)\mathrm{d}s,\\ 
(\mathcal{I}^{-} g)(t) & = \int_{t}^{b}g(s)\mathrm{d}s.
\end{split}\end{equation}
The following discrete operators $\mathcal{I}^{+}_{N}$ and $\mathcal{I}^{-}_{N}$ approximate the operators $\mathcal{I}^{+}$ and $\mathcal{I}^{-}$ as
\begin{equation}\label{sigma1}\begin{split}
(\mathcal{I}^{+}_{N} g)(t) & = \Pi _{N} \textbf{A}^{(1)} V_{N} g(t), \hskip 0.4cm \textbf{A}^{(1)}=h I^{(-1)} D_{N}(\dfrac{1}{\varphi'_{a,b}}),\\
(\mathcal{I}^{-}_{N} f)(t) & = \Pi _{N}\textbf{A}^{(2)} V_{N} g(t), \hskip 0.4cm \textbf{A}^{(2)}=h (I^{(-1)})^{T} D_{N}(\dfrac{1}{\varphi_{a,b}'}).
\end{split}\end{equation}
For a function $ f $, the operator $ \mathcal{F}[f](s) $ is defined by
\begin{equation}\label{laplace}
\mathcal{F}[f](s)=\int_{0}^{c}e^{\frac{-t}{s}}f(t)\mathrm{d}t,
\end{equation}
and it is assumed that  Eq. (\ref{laplace}) is well-defined for some $ c \in [b-a, \infty ] $ and for all $ s $ on the right half of the complex plane, $ \Omega^{+}=\{z \in \mathbb{C } : \Re( z)>0 \}. $

Sinc-convolution method provides formulae of high accuracy and allows $f(s)$ to have an integrable singularity at $s=b-a$ and for $g$ to have singularities at both endpoints of $(a,b)$ \cite{stenger2016handbook}. This property of the sinc-convolution makes it suitable to approximate the weakly singular integral equations.

Now for convenience, some useful theorems related to sinc-convolution method are introduced.
The following theorem predicts the convergence rate of the sinc convolution method.
\begin{theorem}\label{theocon}(\cite{stenger2016handbook})
	(a) Suppose that the integrals $ p(t) $ and $ q(t) $ in (\ref{convints}) exist and are uniformly bounded on  $ (a,b) $, and let $\mathcal{F} $ be defined as (\ref{laplace}). Then the following operator identities hold
	\begin{equation}\label{asli2}
	p=\mathcal{F}(\mathcal{I}^{+}) g , \quad q=\mathcal{F}(\mathcal{I}^{-}) g.
	\end{equation}
	(b) Assume that $ \dfrac{g}{\varphi_{a,b}} \in \mathcal{L}_{\lambda}(\mathcal{D})$. If for some positive $ C '$ independent of $ N $, we have $ \vert \mathcal{F}'(s) \vert \leq C'$ for all $ \Re( s) \geq 0 $, then there is a constant $ C $ which is independent of $ N $ such that 
	\begin{equation} \begin{split}
	\parallel p - \mathcal{F}(\mathcal{I}^{+}_{N}) g \parallel & \ \leq C \sqrt{N} \exp (-\sqrt{\pi \lambda d N}),\\
	\parallel q - \mathcal{F}(\mathcal{I}^{-}_{N}) g \parallel & \ \leq C \sqrt{N}\exp (-\sqrt{\pi \lambda d N}). \end {split}
	 \end{equation}
  \end{theorem}
\subsubsection{Sinc-convolution scheme}
In order to practical use of the convolution method, it is assumed that the dimension of matrices, $ 2N+1 $, is such that the matrices $ \textbf{A}^{(1)} $ and $ \textbf{A}^{(2)} $ are diagonalizable \cite{stenger2016handbook} as follows
\begin{equation}\label{diag1}
\textbf{A}^{(j)}=X^{(j)}S(X^{j})^{-1}, \quad j=1, 2,
\end{equation} 
where 
\begin{equation} \begin{split}
S &= \diag (s_{-N}, \ldots , s_{N} ),\\
X^{(1)}&= [x_{k,l}], \hskip 0.3cm (X^{(1)})^{-1}=[x^{k,l}],\\
X^{(2)}&= [\xi_{k,l}], \hskip 0.3cm (X^{(2)})^{-1}=[\xi^{k,l}].
\end {split}\end{equation}
The integral at Eq.(\ref{asli}) has been split into the following integrals
\begin{equation}\label{split}
\int^b_a \vert t-s \vert ^ {-\lambda} k(t, s, u(s)) \mathrm{d}s=\int^t_a \vert t-s \vert ^ {-\lambda} k(t,s,u(s)) \mathrm{d}s + \int^b_t \vert t-s \vert ^ {-\lambda} k(t,s,u(s)) \mathrm{d}s.
\end{equation}
Based on formulae (\ref{sigma1}), the two discrete nonlinear operators are defined
\begin{equation}
\begin{split}
(\mathcal{K}_{N}^{1}u)(t) & = \Pi _{N} \textbf{A}^{(1)} V_{N}k(t,s,u(s)), \\
(\mathcal{K}_{N}^{2}u)(t) & =\Pi _{N} \textbf{A}^{(2)} V_{N}k(t,s,u(s)).
\end{split}
\end{equation}
The approximate solution takes the form
\[u^{c}_{N}(t)=\sum_{j=-N}^{N}c_{j}\omega_{j}(t),\]
where $ c_{j} $s are unknown coefficients to be determined.
The integrals in right hand side of (\ref{split}) are approximated by the formulae (\ref{sigma1}), (\ref{asli2}) and (\ref{diag1}). We substitute these approximations in (\ref{asli}) and then the approximated equation is collocated at the sinc points. This process reduces solving (\ref{asli}) to solving the following finite dimensional system of equations 
\begin{equation}\label{nonsys}\begin{split}
c_{j}-& \sum_{k=-N}^{N}x_{j,k}\sum_{l=-N}^{N} x^{k,l}\mathcal{F}(s_{k})k(z_{j}, z_{l},c_{l} )\\
- & \sum_{k=-N}^{N}\xi_{j,k}\sum_{l=-N}^{N}\xi^{k,l}\mathcal{F}(s_{k}) k(z_{j}, z_{l}, c_{l}) = y(z_{j}),
\end{split}\end{equation} for $j=-N, \ldots, N.$\\

Equation (\ref{nonsys}) can be expressed in the operator notation as follows
\begin{equation}\label{disequ}
u^{c}_{N} - \mathcal{P}^{c}_{N} \mathcal{K}_{N}^{1}u^{c}_{N} - \mathcal{P}^{c}_{N} \mathcal{K}_{N}^{2}u^c_{N} = \mathcal{P}^{c}_{N}y
\end{equation}
\subsubsection{Convergence analysis}
The convergence analysis of  sinc-convolution method is discussed in this section. The main result is formulated in the following theorem.
\begin{theorem}
	Suppose that $ u(t) $ be an exact solution of Eq. \eqref{asli} and the kernel $ k $ satisfies the Lipschitz condition with respect to the third variable by $ L $. Also, let the assumptions of Theorem \ref{theo3} be fulfilled. Then there is a constant $ C $ which is independent of $ N $ such that 
	\[\Vert u - u^{c}_{N} \Vert \leq C \sqrt{N} \log (N) \exp (-\sqrt{\pi d \lambda N}).
	\]\end{theorem}
\begin{proof}
	By subtracting Eq. \eqref{asli} from Eq. \eqref{disequ},  the following term has been obtained 
	\[
	\Vert u- u^{c}_{N} \Vert \leq \Vert \mathcal{K}^{1}u - \mathcal{P}^{c}_{N}\mathcal{K}_{N}^{1}u^{c}_{N}\Vert + \Vert \mathcal{K}^{2}u - \mathcal{P}^{c}_{N}\mathcal{K}_{N}^{2}u^{c}_{N} \Vert + \Vert y- \mathcal{P}^{c}_{N}y \Vert.
	\]
	Acquiring an upper bound for the first and second terms is almost the same. For this aim, the first term is rewritten as follows
	\[
	\mathcal{K}^{1}u - \mathcal{P}^{c}_{N}\mathcal{K}_{N}^{1}u^{c}_{N} = \mathcal{K}_{1}u - \mathcal{P}^{c}_{N}\mathcal{K}_{1}u^{c}_{N} + \mathcal{P}^{c}_{N}\mathcal{K}_{1}u^{c}_{N} - \mathcal{P}^{c}_{N}\mathcal{K}_{N}^{1}u^{c}_{N},
	\]
	so we have 
	\begin{equation}\begin{split}
	\Vert \mathcal{K}^{1}u - \mathcal{P}_{N}\mathcal{K}_{N}^{1}u^{c}_{N}\Vert & \leq
	\Vert \mathcal{K}^{1}u - \mathcal{K}_{1}u^{c}_{N} \Vert \\
	& + \Vert \mathcal{K}^{1}u^{c}_{N} - \mathcal{P}_{N}\mathcal{K}^{1}u^{c}_{N}\Vert +\Vert \mathcal{P}^{c}_{N} \Vert \Vert \mathcal{K}^{1}u^{c}_{N} - \mathcal{K}_{N}^{1}u^{c}_{N}\Vert,
	\end{split}\end{equation}
	where the second term is bounded by Theorem \ref{interpolationse}. 
	Due to the Lipschitz condition, the first term is bounded by
	\begin{equation*}
	\Vert \mathcal{K}^{1}u - \mathcal{K}^{1}u^{c}_{N} \Vert \leq C_{1} \Vert u - u^{c}_{N} \Vert,
	\end{equation*}
	where $ C_{1} $ is a suitable constant.
	In addition,  Lemma \ref{franks} and Theorem \ref{theocon} help us to find the following upper bound
	\begin{equation}
	\Vert \mathcal{P}^{c}_{N} \Vert \Vert \mathcal{K}_{1}u^{c}_{N} - \mathcal{K}_{N}^{1}u^{c}_{N}\Vert \leq C_{2} \sqrt{N} \log (N) \exp (-\sqrt{\pi d \lambda N}).
	\end{equation}
	Finally, we get
	\begin{equation}
	\Vert u - u^{c}_{N} \Vert \leq  C \sqrt{N} \log (N) \exp (-\sqrt{\pi d \lambda N}).
	\end{equation}
\end{proof}
\section {Numerical experiments}\label{numeric}
This section is devoted to the  numerical experiments concerning the accuracy and the rate of convergence of the presented methods in the paper.
The proposed algorithms are executed  in {Mathematica}$^{\circledR}$. 
 To solve the  nonlinear systems which arise in the formulation of proposed methods,  we have utilized Newton iteration method. In order to find an initial guess for the Newton procedure, the steepest descent method is employed which is less sensitive to the initial guess \cite{burden2001numerical}.  The convergence rate of the sinc convolution and sinc-convolution methods depends on two parameters $\alpha$ and $d$. Specifically, the parameter $d$ demonstrate  the size of the holomorphic domain of $ u $. In all examples, the parameter $\alpha$  is determined by {Theorem} \ref{theo3} and  $d$ is evaluated $3.14$. Furthermore the parameter $ c $ in formula (\ref{laplace}) takes infinity. 
\begin{example}(\cite{pedasvainikko, mandal2019superconvergence})\label{ex:1}
Let us examine the integral equation
\begin{equation}
u(t) - \int_{0}^{1} \vert t - s \vert^{\frac{-1}{2}} u^{2}(s) \mathrm{d}s = g(t), \hskip0.5cm t \in (0, 1),
\end{equation}
where 
\[\begin{split}
g(t) & = [t(1-t)]^{\frac{1}{2}} + \frac{16}{15}t^{\frac{5}{2}} + 2t^{2}(1-t)^{\frac{1}{2}}\\
&+\frac{4}{3}t(1-t)^{\frac{3}{2}} + \frac{2}{5}(1-t)^{\frac{5}{2}}\\
& -\frac{4}{3}t^{\frac{3}{2}} - 2t(1-t)^{\frac{1}{2}} - \frac{2}{3}(1-t)^{\frac{3}{2}},
\end{split}\]
with the exact solution $ u(t) = \sqrt{t(1-t)} $.  The exact solution has singularity near the zero. The numerical results have been shown by { Figure} \ref{figex1}. A reported  in \cite{pedasvainikko},  the maximum of the absolute errors at the collocation points for a piece-wise polynomial collocation method  is around $10^{-7}$ due to the super-convergence property of the piece-wise collocation method.  Furthermore, in \cite{mandal2019superconvergence} the authors  have applied the multi-Galerkin method for weakly singular integral equations of Hammerstein type. A comparison between the reported results reveal a better findings for sinc approach. 
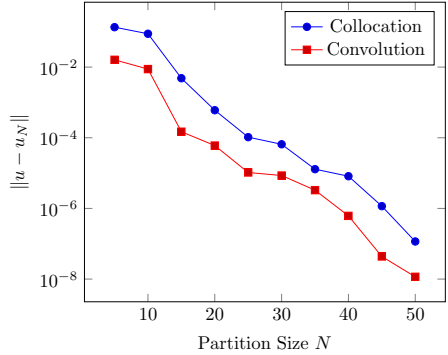
\begin{figure}
	\begin{center}
		\begin{tikzpicture}[scale=0.7]
		\begin{axis}
		[title=, xlabel={Partition Size $N$}, ylabel={$\Vert u-{u}_{N}\Vert$}, legend entries={Collocation,Convolution}, ymode=log, legend pos=north east]
		\addplot coordinates {
		(5,0.1348023655)
		(10,0.08786697)
		(15,0.00484)
		(20,6.02E-04)
		(25,1.04E-04)
		(30, 6.54E-05)
		(35,1.29E-05)
		(40,8.16E-06)
		(45,1.16E-06)
		(50,1.16E-07)
		};
	\addplot coordinates {
		(5,0.016090655)
		(10,0.008786697)
		(15,0.0001484)
		(20,6.02E-05)
		(25,1.04E-05)
		(30,8.54E-06)
		(35,3.29E-06)
		(40,6.16E-07)
		(45,4.36E-08)
		(50,1.16E-08)
	};
		\end{axis}
		\end{tikzpicture}
	\end{center}
	\caption{Plots of the absolute error for sinc-convolution and sinc-collocation methods for Example \ref{ex:1}.  }
	\label{figex1}
\end{figure}
\end{example}
\begin{example}(\cite{tamme2})\label{ex:2}
In this example, we  the following integral equation is consider
\begin{equation}
u(t) - \int_{0}^{1} \vert t - s \vert^{\frac{-1}{4}} u^{2}(s) \mathrm{d}s = g(t), \hskip0.5cm t \in (0, 1).
\end{equation}
The function $ g(t) $ is chosen such that $ u(t) = t^{\frac{3}{2}} $ be the exact solution. The first derivative of the exact solution has singularity near the zero. { Figure} \ref{figex2} shows the error results achieved for the sinc-convolution and sinc-collocation methods which are competetive with the results repoertd in \cite{tamme2}. 
\begin{figure}
	\begin{center}
		\begin{tikzpicture}[scale=0.7]
		\begin{axis}
		[title, xlabel={Partition Size $N$}, ylabel={$\Vert u-{u}_{N}\Vert$}, legend entries={Collocation, Convolution}, ymode=log, legend pos=north east]
		\addplot coordinates {
			(5,0.16090655)
			(10,0.009786697)
			(15,0.001484)
			(20,6.02E-04)
			(25,1.04E-04)
			(30, 9.54E-05)
			(35,1.29E-05)
			(40,8.16E-06)
			(45,1.16E-06)
			(50,1.16E-07)
		};
		\addplot coordinates {
			(5,0.021490655)
			(10,0.0058786697)
			(15,0.00053484)
			(20,1.04E-05)
			(25,6.34E-06)
			(30, 2.34E-06)
			(35,7.29E-07)
			(40,2.36E-07)
			(45,9.26E-08)
			(50,5.726E-09)
		};
		\end{axis}
		\end{tikzpicture}
	\end{center}
	\caption{Plots of the absolute error for sinc-convolution and sinc-collocation methods for Example \ref{ex:2}.}
	\label{figex2}
\end{figure}
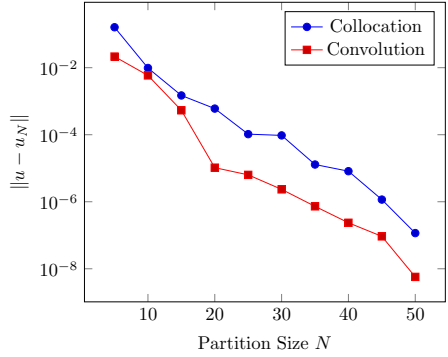
\end{example}
\begin{example}\label{ex3}
	 Consider the integral equation 
	 \begin{equation*}
	 u(t) - \int_{0	}^{1} \vert t-s \vert^{\frac{-1}{2}} \cos(s+u(s))\mathrm{d}s = g(t),
	 \end{equation*}
where $g(t)$ is selected so that $u(t)= \cos(t)$. This example with an infinitely smooth solution is discussed in \cite{MALEKNEJAD20151}.  Here we compare sinc-collocation and sinc-convolution solutions.  Figure \ref{figex3} depicts better results for sinc-convolution approach in comparison with sinc-collocation method.
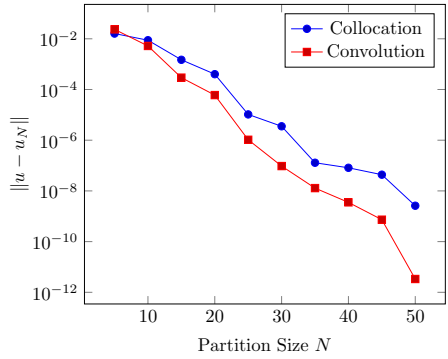
\begin{figure}
	\begin{center}
		\begin{tikzpicture}[scale=0.7]
		\begin{axis}
		[title=, xlabel={Partition Size $N$}, ylabel={$\Vert u-{u}_{N}\Vert$}, legend entries={Collocation,Convolution}, ymode=log, legend pos=north east]
		\addplot coordinates {
			(5,0.016090655)
			(10,0.008786697)
			(15,0.001484)
			(20,4.02E-04)
			(25,1.04E-05)
			(30,3.54E-06)
			(35,1.29E-07)
			(40,8.16E-08)
			(45,4.36E-08)
			(50,2.6E-09)
		};
		\addplot coordinates {
			(5,0.02354985)
			(10,0.00521829)
			(15,0.00029384)
			(20,6.02E-05)
			(25,1.04E-06)
			(30,9.54E-08)
			(35,1.29E-08)
			(40,3.53E-09)
			(45,7.326E-10)
			(50,3.35E-12)
		};
		\end{axis}
		\end{tikzpicture}
		\caption{Plots of the absolute error for sinc-convolution and sinc-collocation methods for Example \ref{ex3}. }
	\end{center}
	\label{figex3}
\end{figure}
\end{example}

\begin{example}\label{ex4}
 In this experiment, we explore the sensitivity of the methods to the
	parameter $\lambda\in\left(  0,1\right)  $ in weakly singular integral equation. We
	consider the equation
	\begin{equation*}
	u(t)  - \int_{0}^{1}\frac{1}{\vert t-s \vert^{1-\lambda}}u^{2}(s)\mathrm{d}s = g(t),
	\end{equation*}
	with the exact solution $u_{\lambda}(t)= t^{2-\lambda}$. We choose $\lambda=\frac{k}{10},$ for $k \in \{1, 2, \dots, 9\} $ and the errors for the sinc-convolution method are reported in Figure \ref{figex4}.
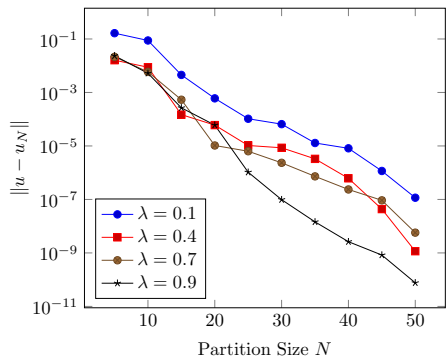
\begin{figure}
\begin{center}
	\begin{tikzpicture}[scale=0.7, transform shape]
	\begin{axis}[title=, xlabel={Partition Size $N$}, log basis x={10},ylabel={$\Vert u -u_{N}\Vert$},ymode=log, legend entries={ $\lambda=0.1$,$\lambda=0.4$,$\lambda=0.7$, $\lambda=0.9$}, legend pos=south west]
	\addplot coordinates {
	(5,0.164723655)
	(10,0.08786697)
	(15,0.00453)
	(20,6.02E-04)
	(25,1.04E-04)
	(30, 6.54E-05)
	(35,1.29E-05)
	(40,8.16E-06)
	(45,1.16E-06)
	(50,1.16E-07)
};
\addplot coordinates {
	(5,0.016090655)
	(10,0.008786697)
	(15,0.0001484)
	(20,6.02E-05)
	(25,1.04E-05)
	(30,8.54E-06)
	(35,3.29E-06)
	(40,6.16E-07)
	(45,4.36E-08)
	(50,1.16E-09)
};

	\addplot coordinates {
	(5,0.021490655)
	(10,0.0058786697)
	(15,0.00053484)
	(20,1.04E-05)
	(25,6.34E-06)
	(30, 2.34E-06)
	(35,7.29E-07)
	(40,2.36E-07)
	(45,9.26E-08)
	(50,5.726E-09)
};
	\addplot coordinates {
	(5,0.02354985)
	(10,0.00521829)
	(15,0.00026632)
	(20,6.02E-05)
	(25,1.04E-06)
	(30,9.81E-08)
	(35,1.45E-08)
	(40,2.63E-09)
	(45,8.326E-10)
	(50,7.63E-11)
};

	\end{axis}
	\end{tikzpicture}
\end{center}
\caption{Plots of the absolute error for sinc-convolution for different values of $\lambda$. }%
\label{figex4}%
\end{figure}


\end{example}

\section*{Conclusion}
In this paper, the sinc-collocation and sinc-convolution methods were considered for nonlinear weakly singular Fredholm integral equations, and rigorous proofs of the exponential convergence of the schemes are obtained.
The theoretical arguments show that direct applying of the collocation method with sinc basis functions leads to the parameter $\xi_{N}$ in the error upper bound.  This parameter is unavoidable due to non-uniformly boundedness of  the sinc interpolation operator. So, the numerical method based on sinc convolution is proposed. It is shown both in theory and numerical experiments that convolution method is more accurate and it achieves exponential convergence with respect to $ N $. The main advantage of the sinc methods for the weakly singular kernels is the fact that they disregard the singularity in  boundaries. The matter is worthy of handling  discrete sinc-convolution operators and extending to the case of full implicit integral equations by utilizing double exponential sinc method.
\section*{Acknowledgement} We are greatly indebted to Professor Frank Stenger (University of Utah) for helpful discussions and remarks. 
\bibliographystyle{acm}
\bibliography{mybibfile}
\end{document}